\documentclass[12pt, a4paper]{article}

\usepackage{amsmath}\multlinegap=0pt
\usepackage[latin1]{inputenc}
\usepackage{amsthm}
\usepackage{amssymb}
\usepackage{graphicx}         
\usepackage[francais,english]{babel}
\usepackage{hyperref}
\numberwithin{equation}{section}
\theoremstyle{plain}
\newtheorem{thm}{Theorem}[section]

\newtheorem{defi}[thm]{Definition}

\theoremstyle{definition}

\theoremstyle{remark}
\newtheorem{rem}[thm]{Remark}

\newcommand{\R}{\mathbb{R}}

\newcommand{\J}{\mathcal{J}}

\newcommand\N{{\mathbb N}}

\newcommand\pref[1]{(\ref{#1})}

\let \eps\varepsilon

\DeclareMathOperator{\argmin}{argmin}

\def\<#1,#2>{\left<#1,#2\right>}

\def\PP{{\cal P}}

\newcommand\gammab{{\overline{\gamma}}}
\newcommand\KL{{\mathrm{KL}}}
\newcommand\MK{{\mathrm{MK}}}
\newcommand\MKb{\overline{{\mathrm{MK}}}}
\newcommand\K{{\mathrm{K}}}
\newcommand\Cm{{\mathrm{C}}_\mu}
\newcommand\Cn{{\mathrm{C}}^\nu}
\newcommand\prox{{\mathrm{prox}}}

\title {Computation of Cournot-Nash equilibria by entropic regularization}
\author {Adrien Blanchet \thanks{\scriptsize  GREMAQ-TSE, Universit\'e de Toulouse, Manufacture des Tabacs
21 all\'ee de Brienne, 31000 Toulouse, FRANCE 
\texttt{Adrien.Blanchet@univ-tlse1.fr}.} \;,\; Guillaume Carlier \thanks{\scriptsize Universit\'e Paris-Dauphine, PSL Research University, CNRS, CEREMADE, 75016, Paris, France and Inria-Paris, MOKAPLAN
\texttt{carlier@ceremade.dauphine.fr}.} \;, \;  Luca Nenna \thanks{\scriptsize Inria-Paris, MOKAPLAN,  2, rue Simone Iff, 75012, Paris, France and CEREMADE, UMR CNRS 7534, Universit\'e Paris IX
Dauphine  \texttt{luca.nenna@inria.fr}.}}

\begin{document}

\maketitle

\begin{abstract}
We consider a class of games with continuum of players where equilibria can be obtained by the minimization of a certain functional related to optimal transport as emphasized in \cite{abgcmor}. We then use the powerful entropic regularization technique to approximate the problem and solve it numerically in  various cases. We also consider the extension to some models with several populations of players. 

\end{abstract}


\textbf{AMS Subject Classifications:}.


\section{Introduction}\label{sec-intro}

There is a long tradition in economics and game theory, since the seminal works of Aumann \cite{Aumann2}, \cite{Aumann}, of considering equilibria in games with a continuum of players, each of whom having a negligible influence on the output of the others.  In particular, Schmeidler \cite{Schmeidler} introduced a notion of non-cooperative equilibrium in games with a continuum of agents and,  Mas-Colell \cite{MasColell} reformulated Schmeidler's analysis in terms of joint distributions over players' actions and characteristics and emphasized the concept of Cournot-Nash equilibrium distributions. There are many examples where such concepts are relevant such as strategic route use in road traffic or networks, social interactions....

\smallskip

The problem can be described as follows: heterogeneous players each have to choose a strategy (or a probability over strategies, i.e. mixed strategies are allowed) so as to minimize a cost, the latter depending on the choice of the whole population of players only through the distribution of their strategies. In other words, on the one hand, each player,  has a negligible influence on the cost. On the other hand, the interactions between players are of mean-field type: it does not matter who plays such and such strategy but rather how many players  chose it. There are  different   mean-field  effects, of different nature and which can be either repulsive (i.e. favoring  dispersion of strategies) or attractive (favoring concentration of strategies). Congestion (the cost of a strategy is higher if it is frequently played) is a typical example of dispersive effect.  In realistic models however, there are also attractive effects: choosing a strategy which is "far" from the strategies played by the other players may be risky or result in some cost.. It should then come as no surprise that, due to such opposite effects,  the analysis of equilibria is complex in general. This explains why, in general, one cannot go much further than proving an existence result, as for instance, following the very elegant approach of Mas-Colell \cite{MasColell}.

\smallskip

More recently, the first two authors \cite{abgcmor} (also see \cite{abgcptrl}) emphasized the fact that for a separable class of costs, Cournot-Nash equilibria can be obtained by the minimization of a certain functional on the set of measures on the space of strategies. This functional typically involves two terms: an optimal transport cost and a more standard integral functional which may capture both congestion and attractive effects (as in \cite{lw1}). Interestingly, this kind of minimization problem is very close to the semi-implicit Euler scheme introduced in the seminal work of \cite{JKO1998} for Wasserstein gradient flows (for which we refer to \cite{AmbrosioGradientFlows2008}).

\smallskip

 The variational approach of  \cite{abgcmor} is somehow more constructive and informative (but less general since it requires a separable cost) than the one relying on fixed-point arguments as in \cite{MasColell} but the optimal transport term cost is delicate to handle. It is indeed costly in general to solve an optimal transport problem and compute an element of the subdifferential of the optimal cost as a function of its marginals. In recent years, however it has been realized that a powerful way to approximate optimal transport is by adding an entropic penalization term. Doing so, the problem becomes projecting for  the Kullback-Leibler divergence a given joint measure on the set of measures with fixed marginals, a task that can be achieved very efficiently by alternate projections (see e.g. Bauschke and Lewis \cite{bauschke-lewis}, Dysktra \cite{Dykstra83}) as shown by Cuturi \cite{CuturiSinkhorn}. This powerful method is intimately related to Sinkhorn algorithm and the Iterated Proportional Fitting Procedure (IPFP), well-known to statisticians and  recently remise au go\^ut du jour by Galichon and Salani\'e \cite{Galichon-Entropic} for the estimation of matching models (we also refer to the recent book of Galichon \cite{galichonbook} for a broader perspective on optimal transport methods, with or without regularization, in economics and econometrics).  Various applications of the IPFP/Sinkhorn algorithm  to optimal transport can be found in \cite{benamouetalentropic}. In order to take advantage of the power of entropic regularization on Wasserstein gradient flows, Peyr\'e \cite{peyregradientflows} introduced an extension of Dykstra's algorithm which he called Dykstra proximal splitting.  It turns out that Peyr\'e's algorithm, recently extended by Chizat et al. \cite{Chizat}, is perfectly well-suited to the computation of Cournot-Nash equilibria as we try to explain in the sequel of the paper and illustrate by various numerical examples. We would also like to emphasize that, in the context of Cournot-Nash equilibrium, entropic regularization is also natural from a theoretical point, it amounts to replace exact cost minimization by some Gibbs-like measure or, equivalently to assume that the cost involves some random term.


\smallskip

The paper is organized as follows. In section \ref{sec-cn}, we recall the concept of Cournot-Nash equilibria, its variational counterpart and the entropic regularization of the latter. In Section \ref{sec-algo}, we describe the proximal  splitting algorithm and the semi-implicit approach. In section \ref{sec-num-res}, we present various numerical results both in dimension one and two and emphasize the influence of the transport cost on the structure of equilibria. Section \ref{sec-several-pop} extends the previous analysis to some models with several populations for which we also present numerical results.


\section{Cournot-Nash  and entropic Cournot-Nash equilibria}\label{sec-cn}

We will restrict ourselves here to the following \emph{finite} Cournot-Nash setting. Not only this will simplify the exposition and enable us to give a simple and self-contained exposition of the variational approach but this will also be consistent with our numerical scheme which anyway considers a finite number of agents' types and a finite number of strategies. We refer to \cite{abgcmor} for the analysis of the continuum case. We consider a population of players, each of whom is characterized by a type which takes values in the type set $X:=\{x_i\}_{i\in I}$ where $I$ is finite. The frequencies of the players' type in the population is given by a probability $\mu:=\{\mu_i\}_{i\in I}$ with $\mu_i\ge 0$ and $\sum_{i=1}^N \mu_i=1$. Each agent has to choose a strategy $y$ from the strategy set $Y:=\{y_j\}_{j\in J}$ with $J$ finite.  The unknown of the problem is a matrix $\gamma :=\{\gamma_{ij}\}_{i\in I, \; j\in J}$ where $\gamma_{ij}$ is the probability that a player of type $x_i$ chooses strategy $y_j$, there is an obvious feasibility constraint on this matrix, obviously it should have nonnegative entries and its first \emph{marginal} should match the given distribution of players $\mu$ i.e.:
\begin{equation}\label{marginalmu}
\sum_{j\in J} \gamma_{ij} =\mu_i,\; \forall i\in I.
\end{equation} 
The matrix $\gamma$ induces a probability $\nu=\Lambda_2(\gamma)=\{\nu_j\}_{j\in J}$ on the set of strategies given by its \emph{second marginal}:
\begin{equation}\label{defnu}
\nu_j :=\sum_{i\in I} \gamma_{ij}, \; \forall j\in J.
\end{equation}
Agents of type $x_i$ who play strategy $y_j$ incur a cost that not only depends on $x_i$ and $y_j$ but also on the whole probability $\nu:=\{\nu_j\}_{j\in J}$ on the strategy space induced by the behavior of the whole population of players, and we denote this cost by $\Psi_{ij}[\nu]$. An equilibrium is then a probability matrix $\gamma$ which is feasible and which is consistent with the cost minimizing behavior of players, which is summarized in the next definition:

\begin{defi}
A Cournot-Nash equilibrium is a matrix $\gamma=\{\gamma_{ij}\}_{i\in I, \; j\in J} \in \R_+^{I\times J}$ which satisfies the feasibility constraint \pref{marginalmu} and such that, defining the strategy marginal $\nu=\Lambda_2(\gamma)$ by \pref{defnu}, one has 
\[ \gamma_{ij}>0 \Rightarrow \Psi_{ij} [\nu]=\min_{k\in J} \Psi_{ik}[\nu].\] 
\end{defi}

Provided  $\Psi_{ij}$ depends continuously on $\nu$, the existence of an equilibrium can easily be proven by Kakutanis' fixed-point theorem, but not much more can be said, at this level of generality. If one further specifies the form of the cost, as we shall do now, following \cite{abgcmor}, one may obtain equilibria by minimizing a certain cost functional.

\subsection{A variational approach to Cournot-Nash equilibria}

We now suppose that the cost $\Psi_{ij}[\nu]$ takes the following separable form
\[\Psi_{ij}[\nu]:=c_{ij}+ f_j(\nu_j)+ \sum_{k\in J} \phi_{kj} \nu_k\]
where $c:=\{c_{ij}\}_{i\in I, \; j\in J} \in \R^{I\times J}$, each function $f_j$ is nondecreasing and continuous, the matrix $\phi:= \{\phi_{kj}\}\in \R^{J\times J}$ is symmetric, i.e. $\phi_{kj}=\phi_{jk}$, 

A possible interpretation of this model is the following: the players represent a population of doctors, their type $x$ represent their region of origin and their $y$ strategy represent the location where they chose to dwell, the total cost of $x_i$-type doctors  is the sum of
\begin{itemize}
\item a transport cost $c_{ij}=c(x_i, y_j)$, 
\item a congestion cost $f_j(\nu_j)$: if location $y_j$ is very crowded i.e. if $\nu_j$ is large, the doctors settling at $y_j$ will see their benefit decrease,
\item an interaction cost with the rest of the population of doctors, one can think that  $\phi_{kj}$ is an increasing function of some distance between $y_k$ and $y_j$ so that $ \sum_{k\in J} \phi_{kj} \nu_k$ represents the average distance to the rest of the population. 
\end{itemize}

The variational approach of \cite{abgcmor} relies on optimal transport, and we shall give a self-contained and simple presentation in the present discrete setting. Firstly it is useful to introduce the marginal maps:
\[\gamma \in \R^{I\times J} \mapsto \Lambda_1(\gamma) =\alpha \in \R^I, \; \alpha_i:= \sum_{j\in J} \gamma_{ij},\]
and
\[\gamma \in \R^{I\times J} \mapsto \Lambda_2(\gamma) =\nu \in \R^J, \; \nu_j:= \sum_{i\in I} \gamma_{ij},\]
as well as
\[\Cm:=\{\gamma=\{\gamma_{ij}\}_{i\in I, \; j\in J} \in \R_+^{I\times J} \; : \; \Lambda_1(\gamma)=\mu\}\]
which is the set of probabilities on $X\times Y$ having $\mu$ as first marginal (recall that $\mu$ is fixed).  For  $\nu=\{\nu_j\}_{j\in J}\in \R_+^J$ such that $\sum_{j\in J} \nu_j=1$, let us also define
\[\Cn:=\{\gamma=\{\gamma_{ij}\}_{i\in I, \; j\in J} \in \R_+^{I\times J} \; : \; \Lambda_2(\gamma)=\nu\}\]
as the set of probabilities on $X\times Y$ having $\nu$ as second marginal. Let us then also define the set of transport plans between $\mu$ and $\nu$ as
\begin{equation}\label{defiplans}
\Pi(\mu, \nu):=\Cm\cap \Cn.
\end{equation}
Given $\nu$ a probability on $Y$, let us define 
\begin{equation}\label{defmk}
\MK(\nu):=\inf_{\gamma \in \Pi(\mu, \nu)} \Big\{ c \cdot \gamma:=\sum_{i,j\in I\times J} c_{ij} \gamma_{ij}\Big\}
\end{equation}
that is the value of the optimal transport problem between $\mu$ and $\nu$ for the cost $c$. Setting
\[\PP(Y):=\{\nu\in \R_+^J \; : \; \sum_{j\in J} \nu_j=1\}\]
consider the optimization problem
\begin{equation}\label{minicn}
\inf_{\nu\in \PP(Y)} \MK(\nu)+E(\nu)
\end{equation}
where the energy $E$ is given by
\begin{equation}\label{defidee}
E(\nu):=\sum_{j\in J} F_j(\nu_j)+\frac{1}{2} \sum_{k,j\in J\times J} \phi_{kj} \nu_k \nu_j
\end{equation}
and $F_j$ is a primitive of the congestion function $f_j$:
\[F_j(t):=\int_0^t f_j(s) \mbox{d}s.\]

We then have 

\begin{thm}
Let $\nu$ solve \pref{minicn} and  $\gamma \in \Pi(\mu, \nu)$ be such that $c\cdot \gamma =\MK(\nu)$, then $\gamma$ is a Cournot-Nash equilibrium. This implies in particular that there exists Cournot-Nash equilibria. 

\end{thm}

\begin{proof}
We have to prove that whenever $\gamma_{ij}>0$ one has
\begin{equation}\label{skonveut}
c_{ij} +f_j(\nu_j)+\sum_{k\in j} \phi_{kj} \nu_k = u_i
\end{equation}
with
\[u_i:=\min_{j\in J} \{c_{ij} +f_j(\nu_j)+\sum_{k\in j} \phi_{kj} \nu_k \}.\]
First observe that $E$ is of class $C^1$ and by construction
\begin{equation}\label{nablaE}
\frac{\partial E} {\partial \nu_j} = f_j(\nu_j)+\sum_{k\in j} \phi_{kj} \nu_k.
\end{equation}
To treat the transport term, $\MK$, we shall recall the classical Kantorovich duality (see \cite{Villani-TOT2003}, \cite{santambook}) as follows. Firstly for $v\in \R^J$ let us define
\[\K(v):=-\sum_{i\in I} \min_{j\in J} (c_{ij}-v_j) \mu_i\]
note that $\K$ is a convex and Lipschitz function whose conjugate, thanks to Kantorovich duality, can be expressed as
\[K^*(\nu)=\MKb(\nu):=\begin{cases} \MK(\nu) \mbox{ if $\nu \in \PP(Y),$}\\ +\infty \mbox{ otherwise.}   \end{cases}\]
Since $\nu$ minimizes $\MKb+E$, one has $0\in \partial \MK(\nu)+\nabla E(\nu)$, setting $v:=-\nabla E(\nu)$, this can be rewritten as $\nu \in \partial \MKb^*(v)=\partial K(v)$ and since $\MKb(\nu)=c\cdot \gamma$ this gives
\[\begin{split}
\MK(\nu)= c\cdot \gamma&= \sum_{j\in \J} v_j \nu_j + \sum_{i\in I} \min_{j\in J} (c_{ij}-v_j) \mu_i\\
&=  \sum_{j\in \J} v_j \nu_j + \sum_{i\in I} u_i \mu_i= \sum_{i,j\in I\times J} (u_i+v_j) \gamma_{ij}.
\end{split}\] 
which, since $u_i+v_j \le c_{ij}$ implies that whenever $\gamma_{ij}>0$, one has $c_{ij}-v_j=u_i$  which is exactly \pref{skonveut}.  This clearly implies the existence of Cournot-Nash equilibria since $\PP(Y)$ is compact and both $\MK$ and $E$ are continuous.
\end{proof}

Note that  if $E$ is convex then the optimality condition $0\in \partial \MKb(\nu)+\nabla E(\nu)$ is necessary and sufficient and there is actually an equivalence between being an equilibrium and being a minimizer in this case.  

\subsection{Entropic regularization}\label{sec-entropic}

Solving \pref{minicn} in practice (even if $E$ is convex) might be difficult because of the transport cost term $\MK$ for which it  is expensive to compute a subgradient. There is however a simple regularization of $\MK$ which is much more convenient to handle: the entropic regularization (see \cite{benamouetalentropic,CuturiSinkhorn,Galichon-Entropic}). Given a regularization parameter $\eps>0$, let us define for every $\nu \in \PP(Y)$:
\[\MK_\eps(\nu):=\inf_{\gamma \in \Pi(\mu, \nu)}  \Big\{ c\cdot \gamma + \eps \sum_{i,j\in I\times J} \gamma_{ij} (\ln(\gamma_{ij})-1)\Big\}.\]
 We then consider the regularization of \pref{minicn} 
\begin{equation}\label{minicne}
\inf_{\nu\in \PP(Y)} \MK_\eps(\nu)+E(\nu)
\end{equation}  
where $E$ is again given by \pref{defidee}. Thanks to the entropic regularization term, \pref{minicne} is a smooth minimization problem which consists in minimizing with respect to $\gamma$ and $\nu$ the objective
\[  c\cdot \gamma + \eps \sum_{i,j\in I\times J} \gamma_{ij} (\ln(\gamma_{ij})-1)+E(\nu)\]
subject to $\gamma_{ij}\ge 0$ (but because of the entropy, these nonnegativity constraints are not binding) and the linear marginal constraints $\gamma \in \Pi(\mu, \nu)$. The first-order optimality conditions give the following Gibbs form for $\gamma_{ij}$:
\begin{equation}\label{gibbs}
\gamma_{ij}=a_i\exp\Big(-\frac{1}{\eps}(c_{ij}+f_j(\nu_j)+\sum_{k\in J} \phi_{kj} \nu_k)\Big) 
\end{equation}
for some $a_i>0$ which has to fulfill the first marginal constraint i.e.
\[a_i= \frac{\mu_i} {\sum_{j\in J} \exp\Big(-\frac{1}{\eps}(c_{ij}+f_j(\nu_j)+\sum_{k\in J} \phi_{kj} \nu_k)\Big) }.\]
Note that these conditions can also be interpreted as a regularized form of a Cournot-Nash equilibrium since they mean that the conditional probabilities on the set of strategies given the players type $\{\frac{\gamma_{ij}}{\mu_i}\}_{j\in J}$ are proportional to $\exp(-\frac{\Psi_{ij}(\nu)}{\eps})$ where $\Psi_{ij}[\nu]=c_{ij}+ f_j(\nu_j)+ \sum_{k\in J} \phi_{kj} \nu_k$ is the total cost incurred by players $x_i$ when choosing strategy $y_j$. Another  equilibrium interpretation (which is customary in economics and econometrics in the framework of discrete choice models) is to consider that the total cost actually contains a random component that is of the form $\eps X_{ij}$ where the $X_{ij}$ are i.i.d. logistic random variables (see \cite{galichonbook}). 

Of course, again when $E$ is convex, since $\MKb_\eps$ is strictly convex, there is a unique minimizer and  the first-order optimality condition for \pref{minicne} is necessary and sufficient so that there is again  equivalence between being a minimizer and a (regularized) Cournot-Nash equilibrium.

\section{A proximal splitting algorithm}\label{sec-algo}

To solve \pref{minicne}, we shall use a proximal splitting scheme using the Kullback-Leibler divergence that was recently introduced by Peyr\'e \cite{peyregradientflows} in  the context of entropic regularization of Wasserstein gradient flows and extended recently by Chizat et al. \cite{Chizat}. First, let us observe that \pref{minicne} can be rewritten as a special instance of a Bregman proximal problem. To see this, let us first rewrite
\[c\cdot \gamma +  \eps \sum_{i,j\in I\times J} \gamma_{ij} (\ln(\gamma_{ij})-1)=  \eps \sum_{i,j\in I\times J} \gamma_{ij}( \ln\Big ( \frac{\gamma_{ij}}{e^{-\frac{c_{ij}}{\eps}}} \Big)-1)\]
which is the same as $\eps \KL(\gamma \vert \gammab)$ where $\gammab_{ij}=e^{-\frac{c_{ij}}{\eps}}$ and $\KL$ is the Kullback-Leibler divergence
\[\KL(\gamma \vert \theta):=  \sum_{i,j\in I\times J} \gamma_{ij}\Big( \ln\Big ( \frac{\gamma_{ij}}  {\theta_{ij}}\Big)-1   \Big), \; \gamma\in \R_+^{I\times J}, \; \theta \in \R_+^{I\times J}.\]
Note that $\KL$ is the Bregman divergence associated to the entropy. Solving \pref{minicne} then amounts to the proximal problem
\begin{equation}\label{proxE}
\prox^{\KL}_G(\gammab)=\argmin_{\gamma \in \R_+^{I\times J} } \Big\{ \KL(\gamma \vert \gammab)+ G(\gamma)\Big\}
\end{equation}
with
\[G(\gamma):=\chi_{\{\Lambda_1 (\gamma) =\mu\}}+\frac{1}{\eps} E(\Lambda_2(\gamma)). \]
Computing directly $\prox^{\KL}_G(\gammab)$ may be an involved task, but the idea of Peyr\'e's splitting algorithm is to express $G$ as a sum of more elementary functionals:
\[G:=\sum_{l=1}^L G_l\]
each of whom being simple in the sense that computing $\prox^{\KL}_{G_l}$ can be done easily (ideally in close form). The algorithm proposed by Peyr\'e generalizes Dykstras' algorithm for $\KL$ projections on the intersection of convex sets and can be described as follows. First extend the sequence of functions $G_1, \cdots, G_L$  by periodicity:
\[G_{l+nL}=G_l, \; l=\{1, \cdots, L\}, \; n \in \N\]
initialize the algorithm by setting the following values for the $I\times J$ matrices
\[\gamma^{(0)}=\gammab, \; z^{(0)}=z^{(-1)}=\cdots =z^{(-L+1)}=e, \; e_{ij}=1, \; (i,j)\in I\times J,\]
and then iteratively define for $n\ge 1$
\begin{equation}\label{dykitergamma}
\gamma^{(n)}=\prox^{\KL}_{G_n}\Big( \gamma^{n-1} \odot z^{(n-L)}\Big) 
\end{equation}
and
\begin{equation}\label{dykiterz}
z^{(n)}= z^{(n-1)} \odot \Big( \gamma^{(n-1)} \oslash \gamma^{(n)}  \Big)
\end{equation}
where $\odot$ and $\oslash$ stand for entry-wise multiplication/division operations:
\[(\gamma \odot \theta)_{ij}=\gamma_{ij}\theta_{ij}, \; (\gamma \oslash \theta)_{ij}=\frac{ \gamma_{ij}}{\theta_{ij}}.\]
We refer to \cite{peyregradientflows} and \cite{Chizat} for the convergence of this algorithm under suitable assumptions (convexity of the functions $G_l$ and a certain qualification condition), the idea being that at the level of the dual problem, which is smooth, this algorithm amounts to perform an alternate block minimization.

\subsection{A class of convex problems}

Note that the congestion term $\sum_{j\in J} F_j(\nu_j)$ is convex because $f_j$ is nondecreasing, but the quadratic  interaction energy $\nu\mapsto \sum_{j,k\in J\times J} \phi_{kj} \nu_k \nu_j$ is in general not convex. However, using Cauchy-Schwarz inequality, it satisfies
\[ \sum_{j,k\in J\times J} \phi_{kj} \nu_k \nu_j \ge -\Big( \sum_{j,k\in J\times J} \phi_{kj}^2 \Big)  \sum_{j\in J}  \nu_j^2\]
so that if $F_j$ is $1$-strongly convex:
\[F_j(t)=\frac{1}{2} t^2 + H_j(t)\]
with $H_j$ convex and 
\begin{equation}  \label{norminter}
\sum_{j,k\in J\times J} \phi_{kj}^2 <1,
\end{equation}
then $E$ is convex  as the sum $E=E_2+E_3$ of the convex quadratic term 
\[E_2(\nu):=\frac{1}{2}  \sum_{j\in J}  \nu_j^2 +  \frac{1}{2} \sum_{k,j\in J\times J}  \phi_{kj} \nu_k \nu_j\]
and the remaining convex congestion term
\[E_3(\nu):=\sum_{j\in J} H_j(\nu_j).\]
In this setting one can write \pref{minicne} as 
\[\inf_{\gamma\in \R_+^{I\times J}} \Big\{\KL(\gamma\vert \gammab)+G_1(\gamma)+G_2(\gamma)+G_3(\gamma)\Big\}\]
where 
\[G_1(\gamma)=\chi_{\{\Lambda_1 (\gamma) =\mu\}}=\begin{cases} 0 \mbox{ if $\Lambda_1( \gamma) =\mu$}\\ +\infty \mbox{ otherwise }\end{cases}\]
and 
\[G_2=\frac{1}{\eps} E_2 \circ \Lambda_2, \; G_3=\frac{1}{\eps} E_3 \circ \Lambda_2.\]
To implement the proximal splitting scheme  \pref{dykitergamma}-\pref{dykiterz} in this case, one has to be able to compute the three proximal maps $\prox^{\KL}_{G_l}$ with $l=1, 2, 3$. The proximal map of $G_1$ corresponds to the fixed marginal constraint $\Lambda_1 (\gamma)=\mu$, it is well-known and it  is given in closed form as:
 \[ \Big(\prox_{G_1}^{\KL}(\theta)\Big)_{ij}=  \frac{ \mu_i \theta_{ij} }   { \sum_{k\in J} \theta_{ik}}.\]
Given $\theta\in \R_+^{I\times J}$, $\gamma:=\prox_{G_2}^{\KL}(\theta)$ is of the form
\[\gamma_{ij}=\theta_{ij} \exp\Big( -\frac{\nu_j+\sum_{k\in J} \phi_{kj} \nu_k}{\eps} \Big)\]
where $\nu$ denotes the second marginal of $\gamma$, so that summing over $i$, $\nu$ is obtained by solving the system:
\[\nu_j=\Big(\sum_{i\in I} \theta_{ij}\Big)  \exp\Big( -\frac{\nu_j+\sum_{k\in J} \phi_{kj} \nu_k}{\eps} \Big)\]
which, when \pref{norminter} holds, can be solved in practice in a few Newton's steps.  The computation of $\gamma:=\prox_{G_3}^{\KL}(\theta)$ is simpler, setting $h_j:=H'_j$ the first-order equation first leads to 
\[\gamma_{ij}=\theta_{ij} \exp\Big( -\frac{h_j(\nu_j)}{\eps} \Big)\]
and the $\nu_j$'s are obtained by solving 
\begin{equation}\label{proxconges}
\nu_j=\Big(\sum_{i\in I} \theta_{ij}\Big)  \exp\Big( -\frac{h_j(\nu_j)}{\eps} \Big)
\end{equation}
which is a separable system of monotone equations, which we shall again solve by Newton's method.

\subsection{A semi-implicit scheme for more general nonconvex cases}\label{sec-semi} 

We now go back to the general case where $E$ is not necessarily convex because of the interaction term given by the symmetric matrix $\phi_{kj}$. Even though there is no theoretical convergence guarantee (but if the following scheme converges, it converges to an equilibrium), the semi-implicit scheme which we now describe gives good results in practice. The idea is simple and consists in replacing the nonconvex interaction term by its linearization. More precisely, we will approximate our initial problem \pref{minicne}:
\begin{equation}
\inf_{\nu\in \PP(Y)} \MK_\eps(\nu)+E(\nu)
\end{equation}
where $E$ is the sum of the convex congestion cost and the nonconvex quadratic interaction cost, by a succession of convex problems, starting from $\nu^{0}\in \PP(Y)$, iteratively solve for $n\ge 1$ 
\begin{equation}\label{semi-imp}
\nu^{(n+1)} = \argmin_{\nu\in \PP(Y)} \MK_\eps(\nu)+E^{(n)}(\nu)
\end{equation}
where  in $E^{(n)}$ we have linearized the interaction term:
\[E^{(n)}(\nu)=\sum_{j\in J} F_j(\nu_j) +\sum_{j\in J} V_j^{(n)} \nu_j, \; V_j^{(n)}:=\sum_{k\in J} \phi_{kj} \nu_k^{(n)}.\]
Of course, we can solve \pref{semi-imp} by the Dykstra proximal-splitting scheme described in the previous paragraph. More precisely, the linear term can be absorbed by the $\KL$ term so that we only have two proximal steps: one corresponding to the (explicit) projection fixed marginal constraint and one corresponding to the congestion cost (corresponding to  \pref{proxconges} using $f_j$ instead of $h_j$).

\section{Numerical results}\label{sec-num-res}
We now present some numerical results in dimension $d=1$ and $d=2$.
As we have pointed out in section \ref{sec-entropic}, the strength of the entropic regularization, and consequently of the Dykstra's algorithm, lies in the fact that we can treat
optimal transportation problems with any transport cost, in particular both concave and convex cost functions can be considered.
Thus, if we consider the cost $c_{ij}=|x_i-y_j|^p$ with $p>0$ (convex cost if $p>1$ and concave otherwise), one can analyze
how the shape of the unknown marginal $\nu$ changes by varying the exponent $p$.
Before showing the results, we want to focus on an other aspect of the entropic regularization, namely diffusion.
Indeed, once we add the entropic term to the optimal transport term, then this regularization spreads the support of the plan $\gamma$ and defines a strongly convex problem with a unique solution.
So it is interesting to see how the support of the optimal $\gamma$ varies by decreasing the parameter $\eps$.
Let us consider the standard quadratic cost $c_{ij}=|x_i-y_j|^2$ and the following energy $E(\nu)$
\begin{equation}
\label{energy1}
    E(\nu)=\sum_{j\in J}\nu_j^8+\frac{1}{2}\sum_{k,j\in J\times J}\phi_{kj}\nu_k\nu_j+\sum |y_j-9|^4,
\end{equation}
where $\phi_{kj}=10^{-4}|y_k-y_j|^2$ and the third term is a confinement potential.
We notice that there is no need to compute a proximal step for the potential, indeed it can be absorbed by the $\KL$ term.
We know that in this case the optimal $\gamma$ (for instance see \cite{abgcmor}) is a pure Cournot-Nash equilibrium, which actually means that $\gamma$ has the form $\gamma_T=({\mathrm{id}},T)_\#\mu$
where $T$ is the optimal map.
In Figure \ref{fig1} we plot the support of the optimal $\gamma$ and its marginal $\nu$ for different values of $\eps$.
As expected the support of the regularized $\gamma$ concentrates on the graph of $T$ as $\eps$ decreases.

\begin{figure}[htbp]

\begin{tabular}{@{}c@{\hspace{0.5mm}}c@{\hspace{0.5mm}}c@{\hspace{0.5mm}}c@{\hspace{0.5mm}}c@{\hspace{0.5mm}}@{}}

\centering
\includegraphics[ scale=0.245]{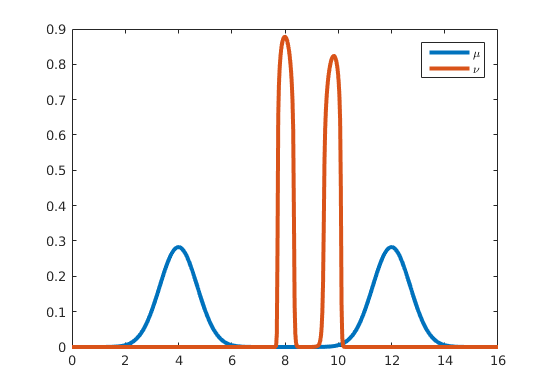}&
\includegraphics[ scale=0.23]{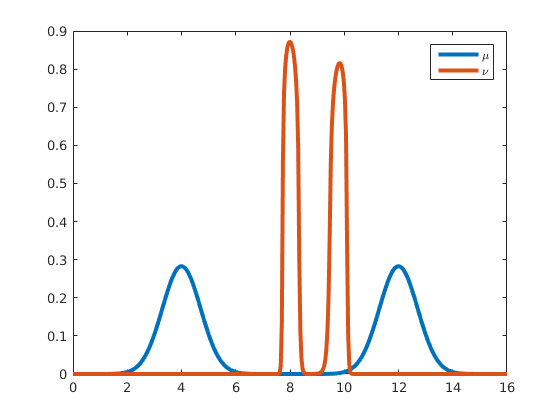}&
\includegraphics[ scale=0.23]{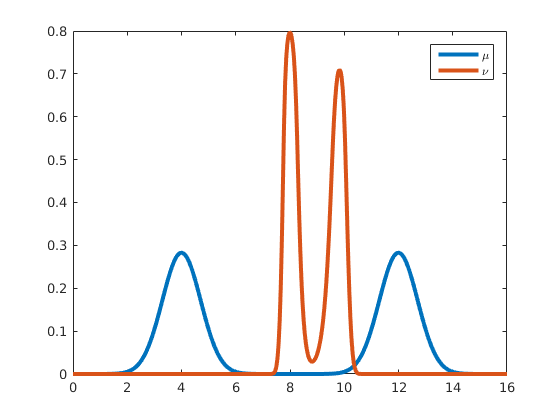}&

\includegraphics[ scale=0.23]{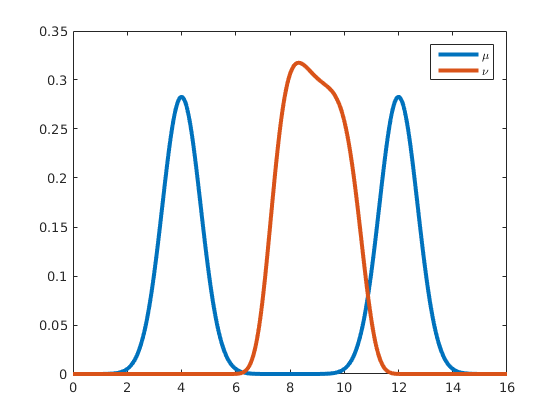}&
\\
$\eps=0.05$ & $\eps=0.1$ & $\eps=0.5$ & $\eps=10$\\
\includegraphics[ scale=0.1]{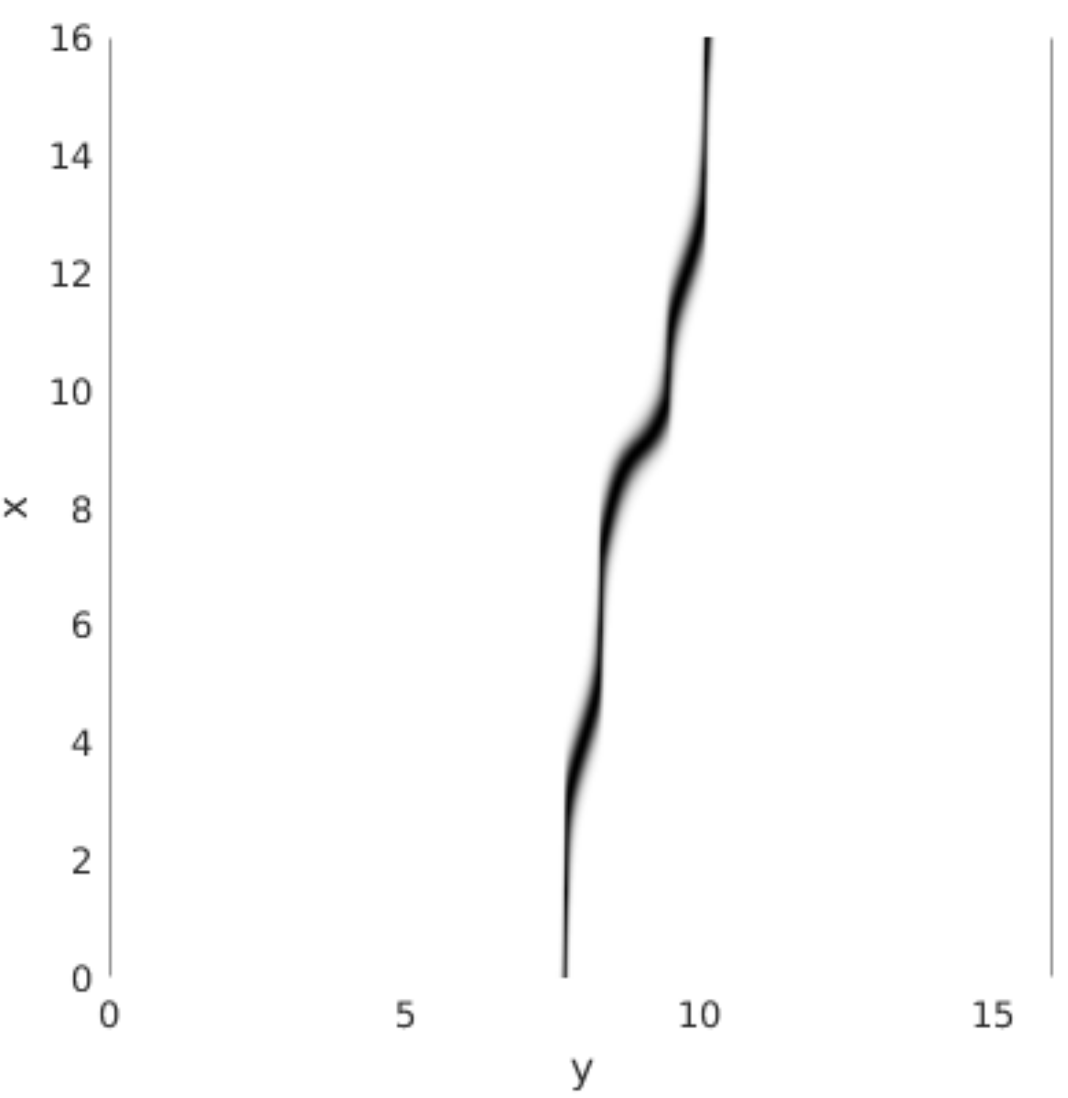}&
\includegraphics[ scale=0.105]{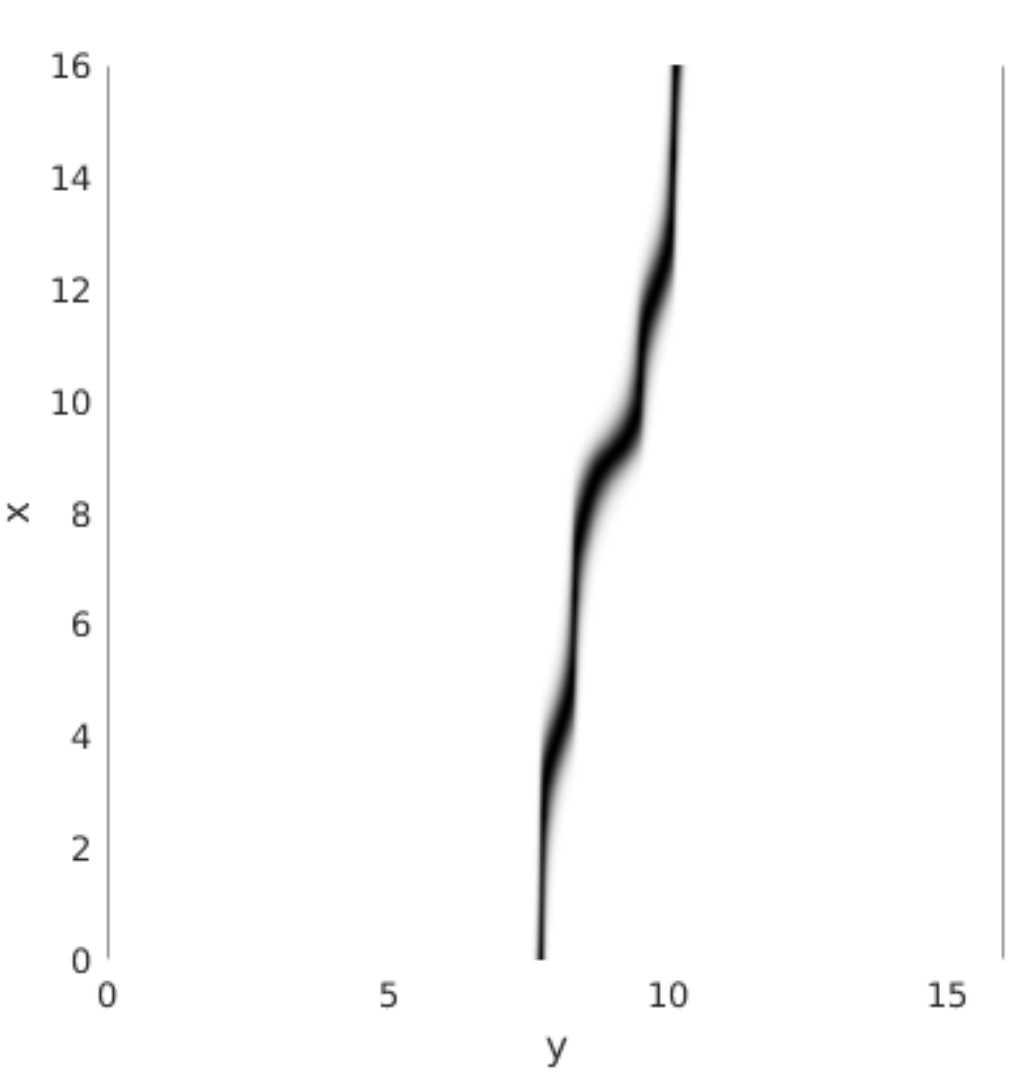}&
\includegraphics[ scale=0.1]{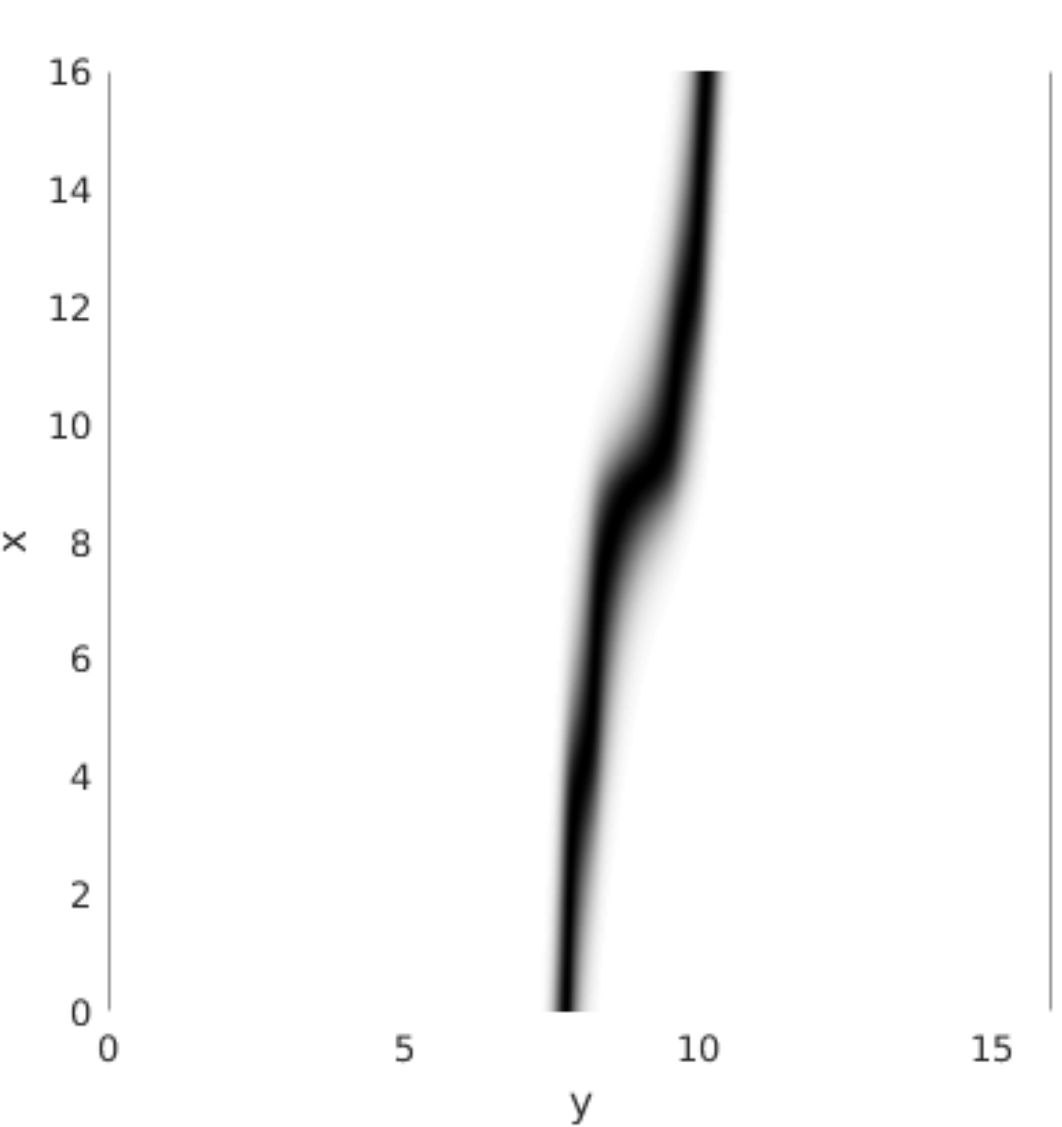}&

\includegraphics[ scale=0.1]{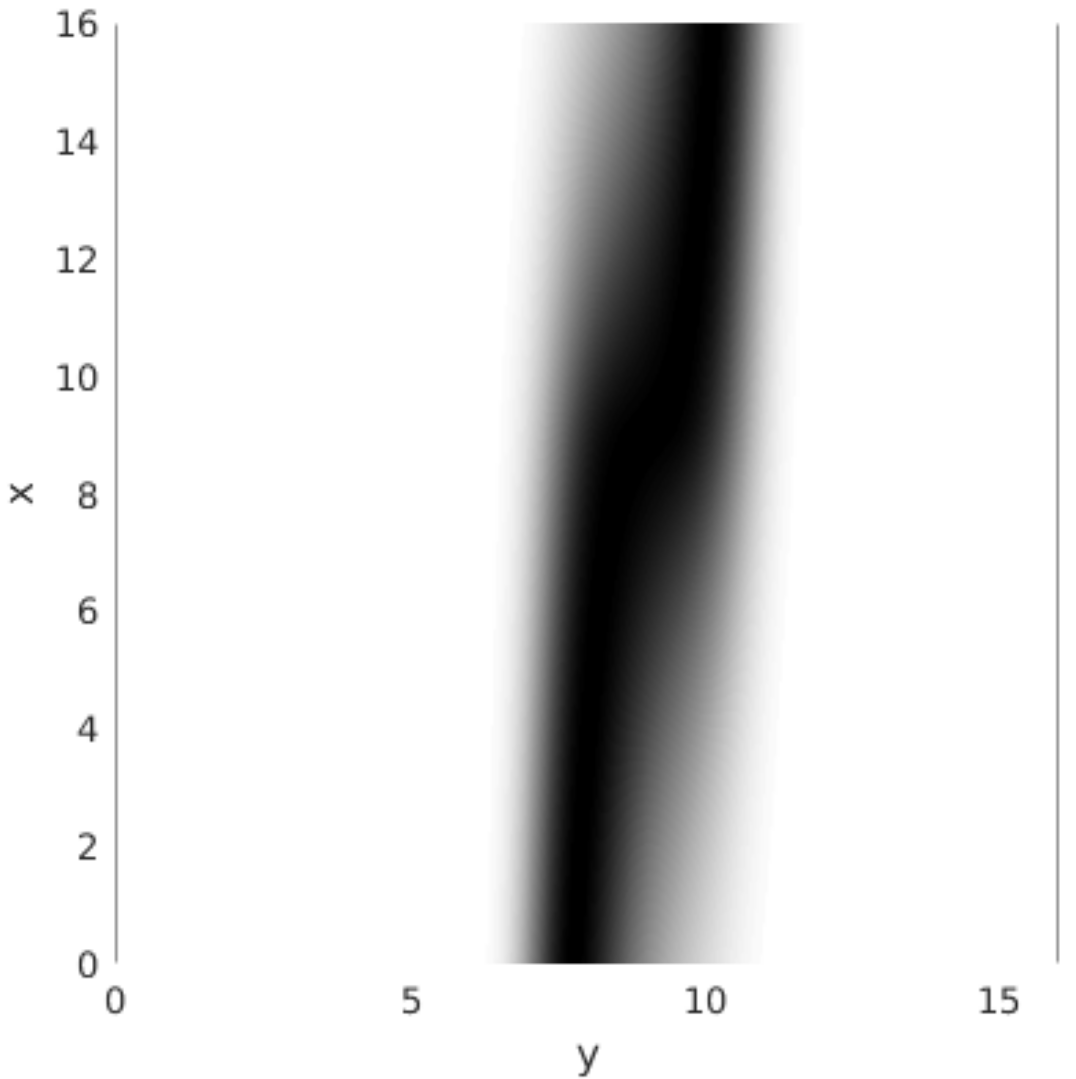}&
\\
$\eps=0.05$ & $\eps=0.1$ & $\eps=0.5$ &  $\eps=10$\\

\end{tabular}
\caption{\textit{Top: The initial distribution $\mu$ (blue solid line) and the solution $\nu$ (red solid line) for $\eps\in\{0.05,0.1,,0.5,10\}$.
         Bottom: The support of $\gamma$ for $\eps\in\{0.05,0.1,,0.5,10\}$.}}
\label{fig1}
\end{figure}
In Section \ref{sec-semi}, we have pointed out that a semi-implicit approach can be applied in order to treat an energy $E$ which is not convex.
We want, now, to compare the performances of the implicit and the semi-implicit approach in terms of \textit{CPU time} and \textit{number of iterations} when 
$\eps$ varies. By looking at  Figure \ref{fig2}, we notice the number of iterations, as well as the CPU time, of the semi-implicit approach are smaller than
the ones for the implicit approach. This is quite obvious as in the semi-implicit scheme, the interaction term can be absorbed by the $\KL$ term so that one has to compute only two proximal steps instead of three. 

\begin{figure}[htbp]
\label{fig2}
\begin{tabular}{@{}c@{\hspace{1.5mm}}c@{\hspace{1.5mm}}@{}}

\centering
\includegraphics[ scale=0.47]{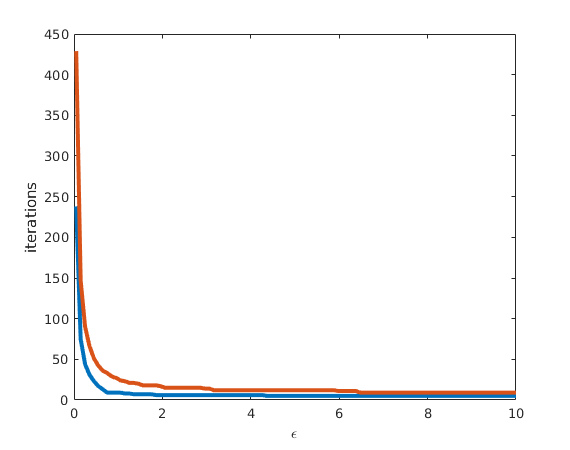}&
\includegraphics[ scale=0.5]{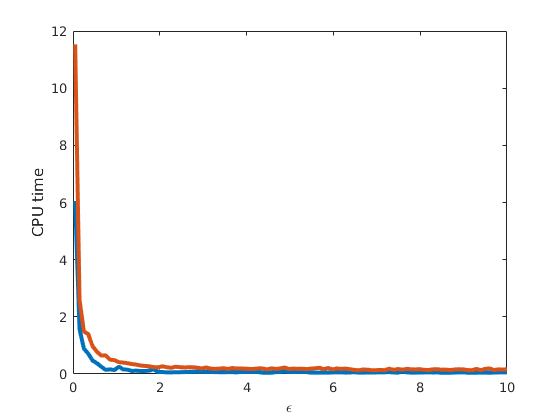}
\\
iterations & CPU time in seconds \\

\end{tabular}
\caption{\textit{Left: the number of iterations for the semi-implicit (blue) and for the implicit (red).
         Right: CPU time for the semi-implicit (blue) and for the implicit (red).}}
\label{fig2}
\end{figure}

\subsection{Dimension one}
Let us first consider the one-dimensional case.
One of the main advantages of the scheme we have proposed is that we can consider any kind of cost function.
Thus, take  $c_{ij}=|x_i-y_j|^p$ and the energy $E$ given by (\ref{energy1}), then we want to visualize the optimal $\nu$ as
$p\in(0,M]$ with $M$ large.
For the simulations in Figure \ref{fig1D_1},  we have used a $N=500$ grid points discretization of $[0,16]$ and we have treated the interaction term with a semi-implicit approach.
Then, we have chosen the smallest $\eps$ possible for each cost function tested.
As one can notice for $p\leq 1$ the optimal $\nu$ has a connected support whereas for $p>1$, the support of $\nu$ is closer to the one of $\mu$.
Finally, we obtain an optimal $\nu$ which tends to be concentrated near $y=9$ due to the external potential, except for large $p$ where the optimal transport term
becomes \textit{dominant} so that the second marginal $\nu$ tends  to be \textit{close}  to the initial distribution $\mu$.  

\begin{figure}[htbp]
\label{fig1D_1}
\centering
\advance\leftskip-2cm
\begin{tabular}{c@{\hspace{0.5mm}}c@{\hspace{0.5mm}}c@{\hspace{0.5mm}}}

\centering
\includegraphics[ scale=0.3]{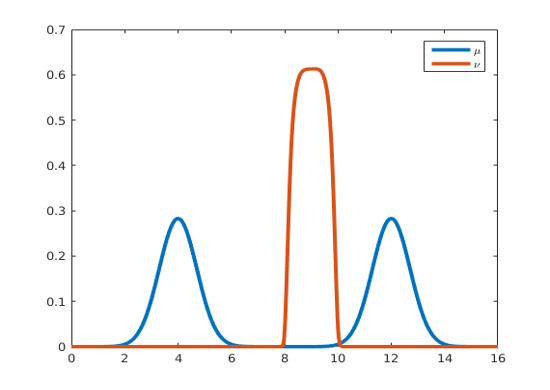}&
\includegraphics[ scale=0.3]{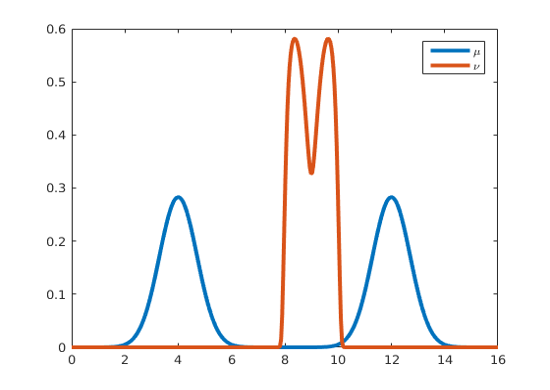}&
\includegraphics[ scale=0.4]{1D_2bumps_pow8_quadratic_int_p2.png}\\
$p=0.1$ & $p=1$ & $p=2$\\
\includegraphics[ scale=0.4]{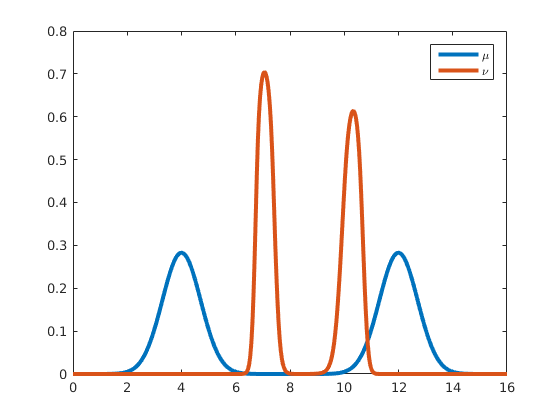}&
\includegraphics[ scale=0.4]{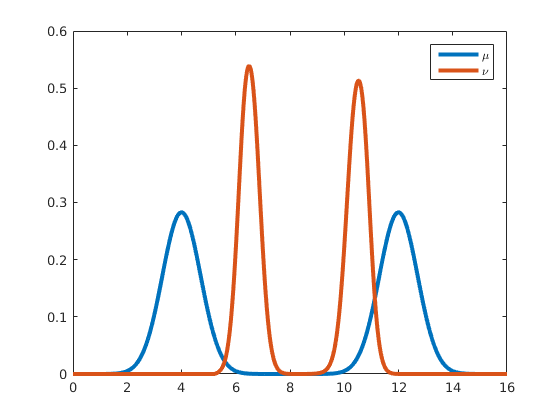}&
\includegraphics[ scale=0.4]{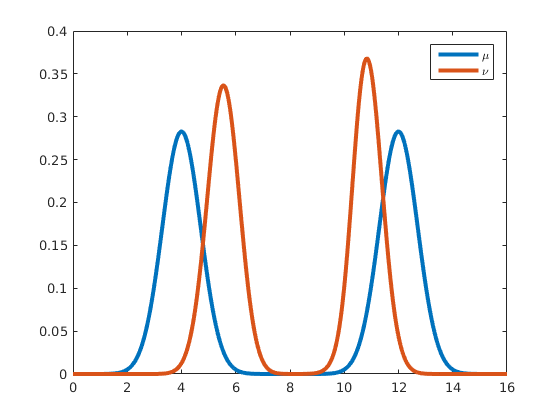}\\
$p=3$& $p=4$&  $p=8$\\
\includegraphics[ scale=0.4]{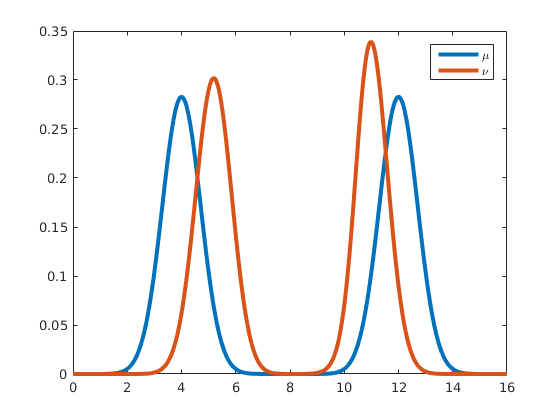}&
\includegraphics[ scale=0.4]{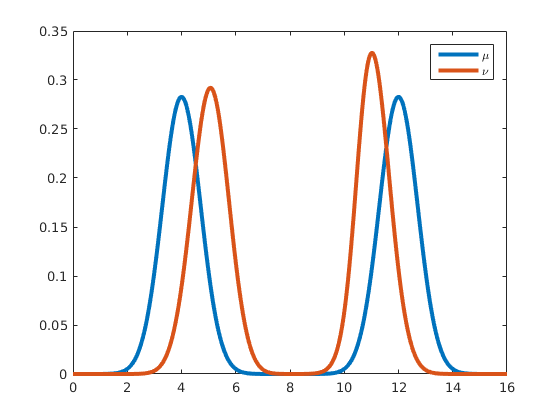}&
\includegraphics[ scale=0.4]{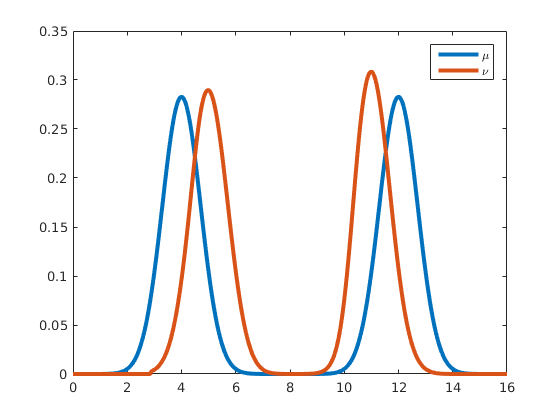}\\
$p=16$& $p=32$&  $p=64$\\

\end{tabular}
\caption{\textit{The initial distribution $\mu$, a sum of two translated Gaussian, (blue solid line) and the solution $\nu$ (red solid line) for $p\in\{0.1,1,2,3,4,8,16,32,64\}$.}}
\label{fig1D_1}
\end{figure}



Let us now consider an energy $E$ given by
\begin{equation}
E(\nu)= \sum_{j\in J} \ln(\nu_j)+\sum_{k,j\in J\times J}\phi_{kj}\nu_k\nu_j +\sum_{j\in J}(y_j-5)^3
\end{equation}
where $\phi_{kj}$ is a cubic interaction $\phi_{kj}=10^{-4}|x_i-y_j|^3$. 
The simulations are presented in Figures \ref{fig1D_2} and \ref{fig1D_3} for different initial distribution: a uniform density on $[0,1]$ and the sum of two translated Gaussians, respectively.
For both the numerical experiments we have used  $N=500$ grid points discretization of $[0,10]$ and treated the interaction term with a semi-implicit approach.
One can observe, as in the previous case, that the structure of the optimal $\nu$  becomes close to the one of the initial ditribution as $p$ increases. 

\begin{figure}[htbp]
\label{fig1D_2}
\advance\leftskip-2cm
\begin{tabular}{c@{\hspace{0.5mm}}c@{\hspace{0.5mm}}c@{\hspace{0.5mm}}}

\centering
\includegraphics[ scale=0.4]{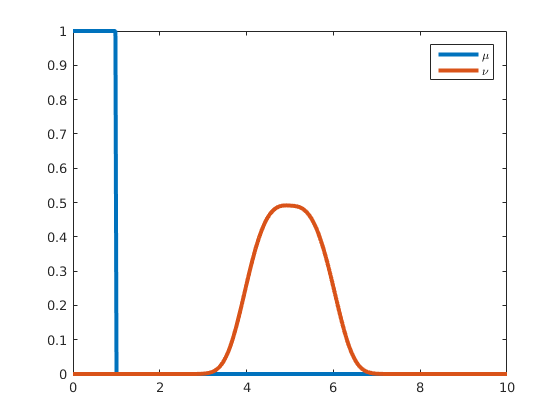}&
\includegraphics[ scale=0.3]{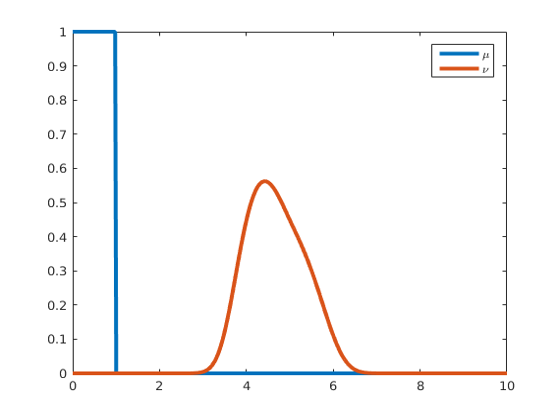}&
\includegraphics[ scale=0.4]{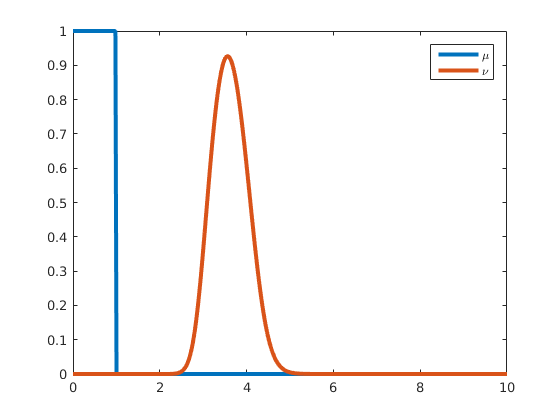}\\
$p=0.1$ & $p=1$ & $p=2$\\
\includegraphics[ scale=0.4]{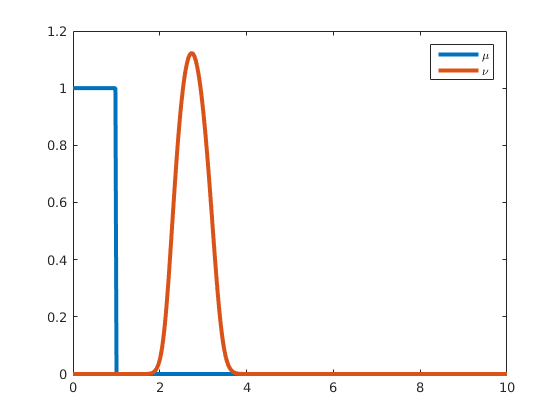}&
\includegraphics[ scale=0.4]{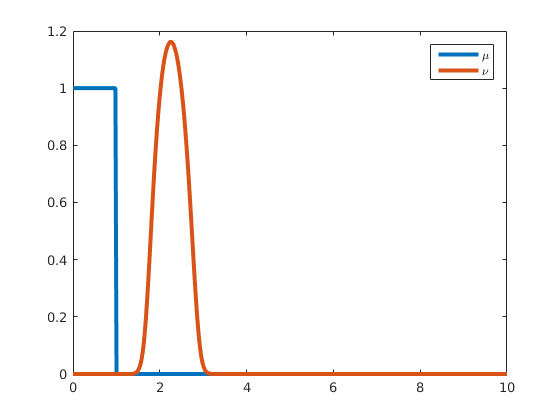}&
\includegraphics[ scale=0.3]{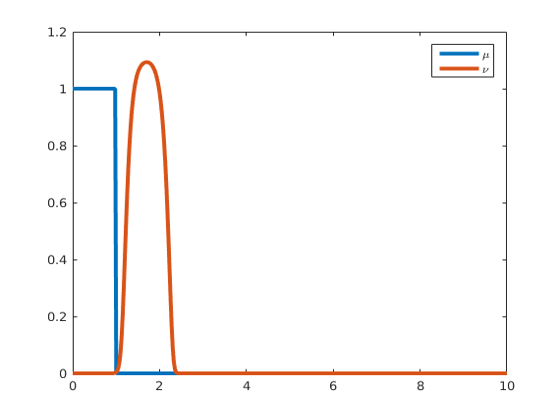}\\
$p=3$& $p=4$&  $p=8$\\
\includegraphics[ scale=0.4]{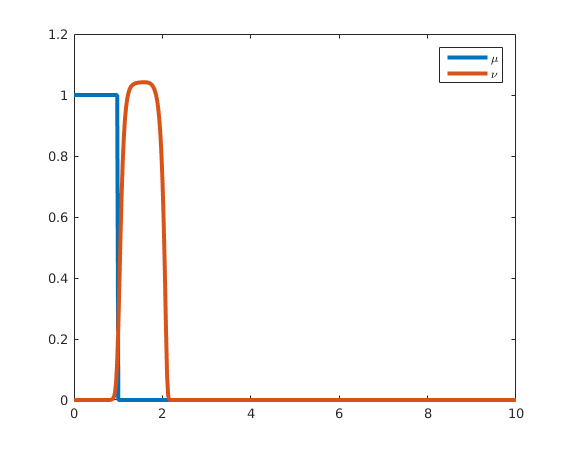}&
\includegraphics[ scale=0.4]{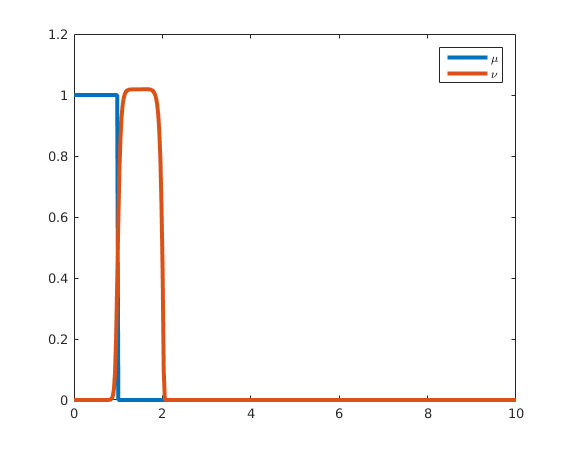}&
\includegraphics[ scale=0.4]{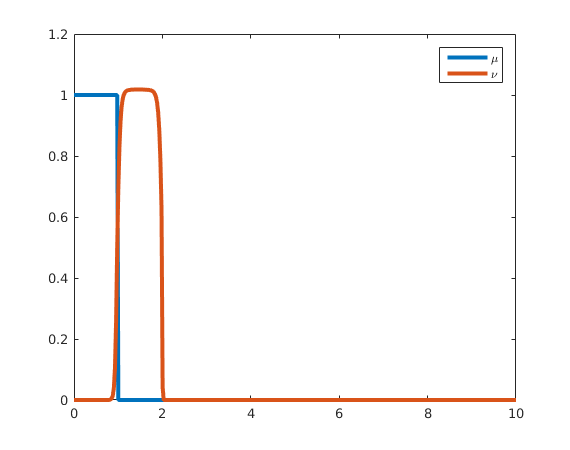}\\
$p=16$& $p=32$&  $p=64$\\

\end{tabular}
\caption{\textit{The initial distribution $\mu$, a uniform density on $[0,1]$, (blue solid line) and the solution $\nu$ (red solid line) for $p\in\{0.1,1,2,3,4,8,16,32,64\}$.}}
\label{fig1D_2}
\end{figure}

\begin{figure}[htbp]
\label{fig1D_3}
\advance\leftskip-2cm
\begin{tabular}{@{}c@{\hspace{0.5mm}}c@{\hspace{0.5mm}}c@{\hspace{0.5mm}}@{}}

\centering
\includegraphics[ scale=0.4]{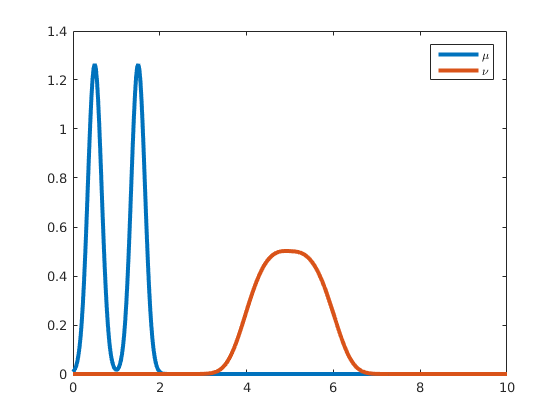}&
\includegraphics[ scale=0.4]{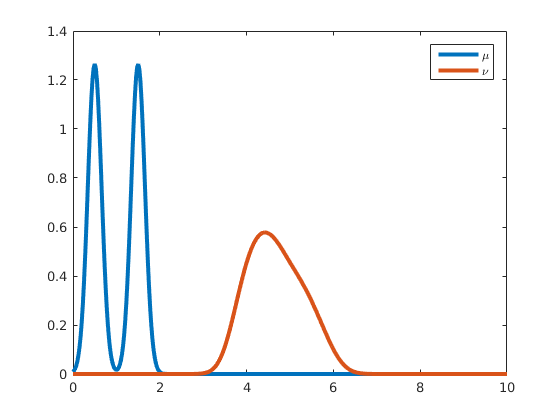}&
\includegraphics[ scale=0.4]{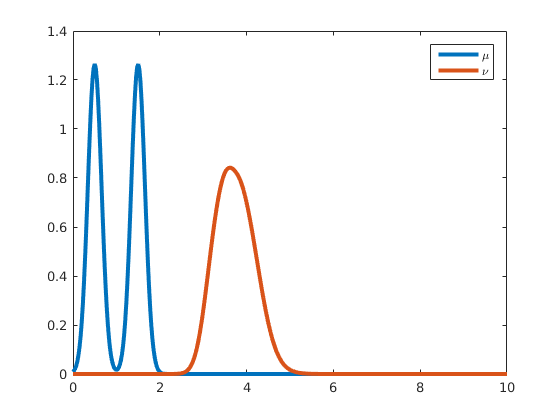}\\
$p=0.1$ & $p=1$ & $p=2$\\
\includegraphics[ scale=0.4]{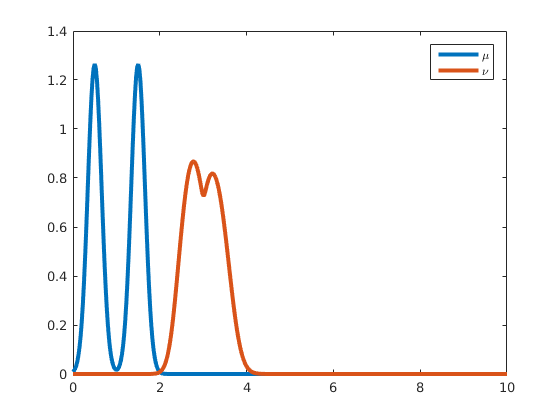}&
\includegraphics[ scale=0.4]{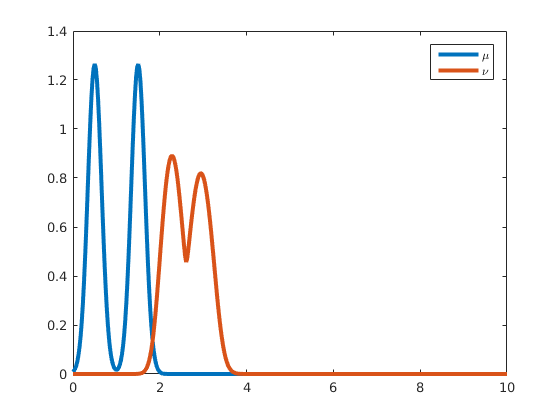}&
\includegraphics[ scale=0.4]{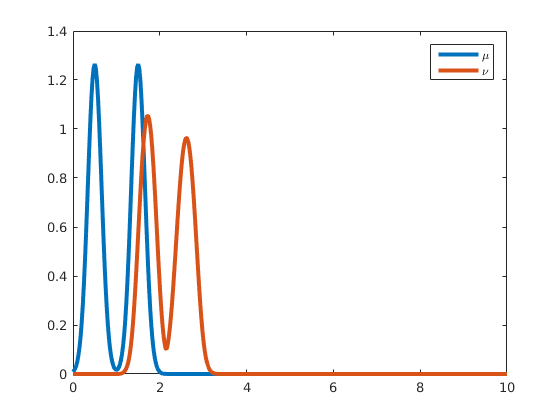}\\
$p=3$& $p=4$&  $p=8$\\
\includegraphics[ scale=0.4]{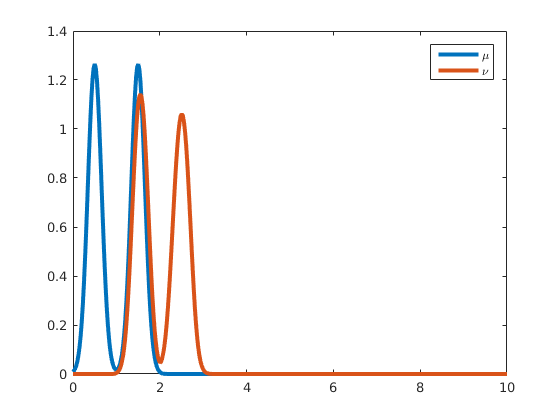}&
\includegraphics[ scale=0.4]{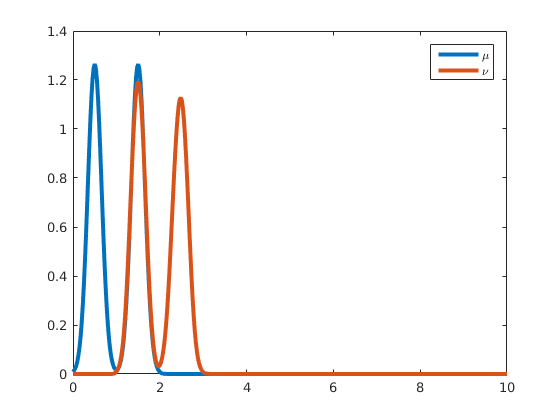}&
\includegraphics[ scale=0.4]{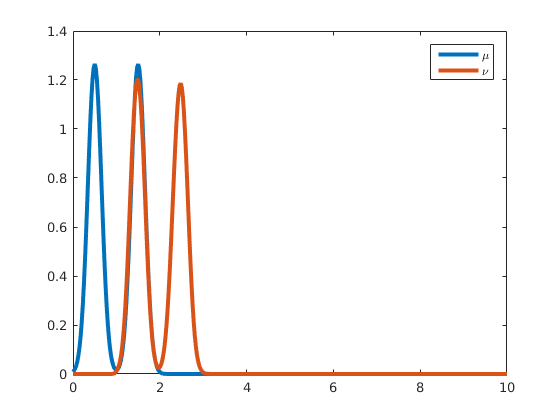}\\
$p=16$& $p=32$&  $p=64$\\

\end{tabular}
\caption{\textit{The initial distribution $\mu$, a sum of two translated Gaussians, (blue solid line) and the solution $\nu$ (red solid line) for $p\in\{0.1,1,2,3,4,8,16,32,64\}$.}}
\label{fig1D_3}
\end{figure}

\subsection{Dimension two} 

For the $2d$ case, we always take $c(x,y)=\|x-y\|^p$, a congestion of the form $F_j(\nu_j)=\nu_j^8$, quadratic interactions $\phi_{kj}=10^{-4}\|y_k-y_j\|^2$ and 
a potential $v_j=\|y_j-3\|^4$.
The simulations in Figure \ref{fig5} are obtained by using a $N\times N$, with $N=80$, discretization of $[0,5]^2$ and by treating the interaction term with a semi-implicit approach.
As in the $1-$dimensional case, we notice the same effect on the support of $\nu$ when we make $p$ vary.
\begin{figure}[htbp]
\label{fig5}
\advance\leftskip-3.5cm
\begin{tabular}{@{}c@{\hspace{0.5mm}}c@{\hspace{0.5mm}}c@{\hspace{0.5mm}}c@{\hspace{0.5mm}}@{}}

\centering
\includegraphics[ scale=0.28]{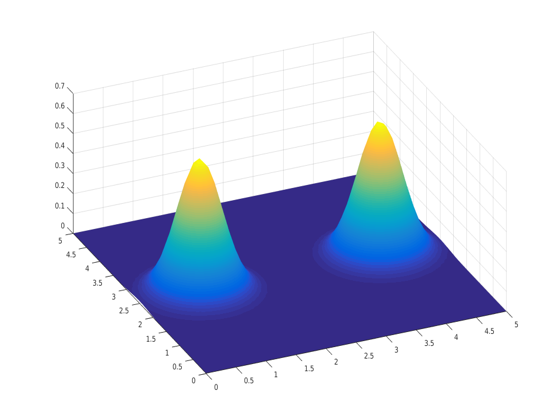}&
\includegraphics[ scale=0.32]{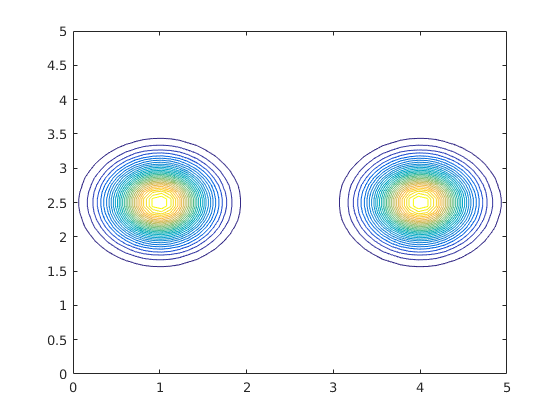}&
\includegraphics[ scale=0.28]{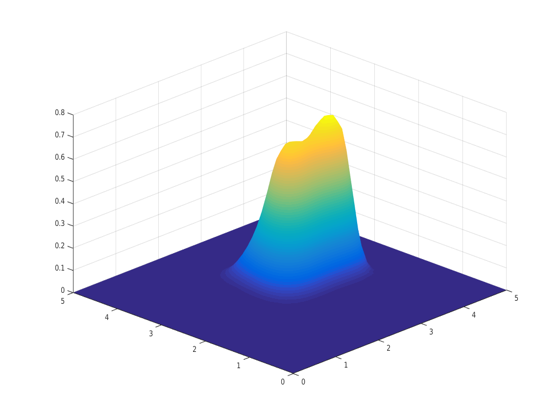}&
\includegraphics[ scale=0.32]{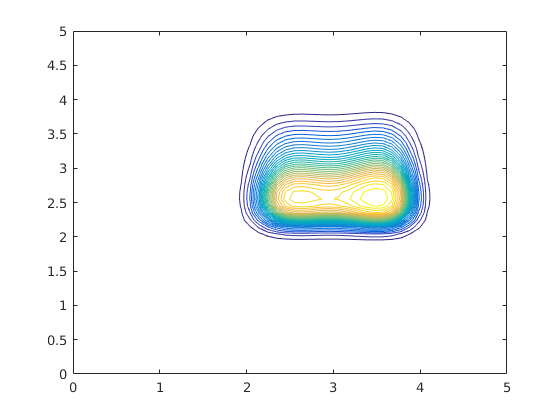}\\
surface plot  of $\mu$ & support of $\mu$ & surface plot of $\nu$ for $p=0.5$ $$ & support of $\nu$ for $p=0.5$\\

\includegraphics[ scale=0.28]{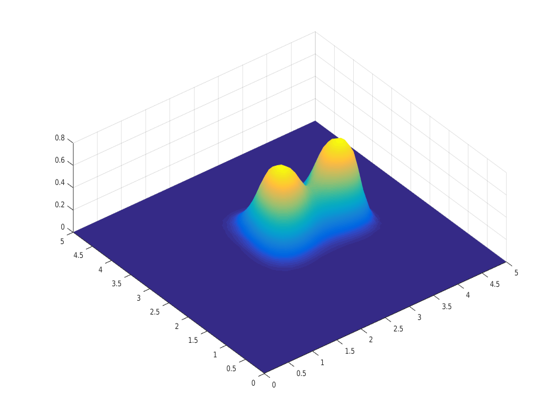}&
\includegraphics[ scale=0.32]{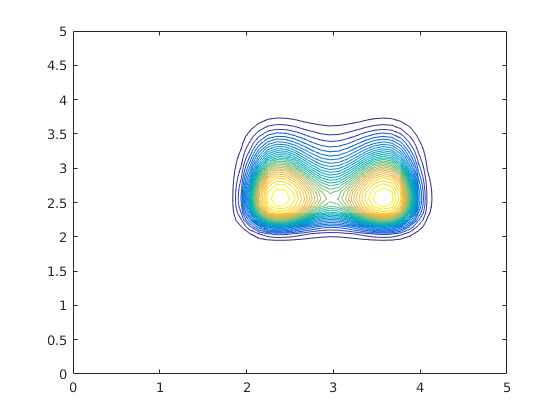}&
\includegraphics[ scale=0.28]{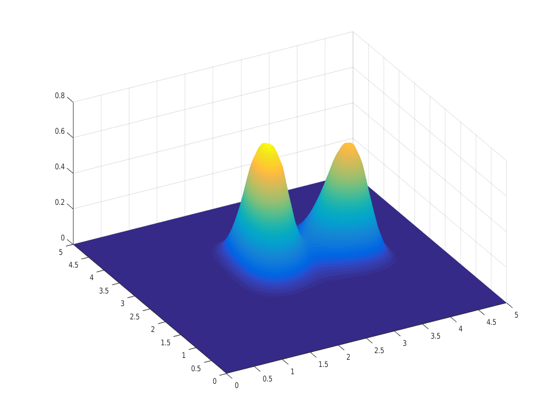}&
\includegraphics[ scale=0.32]{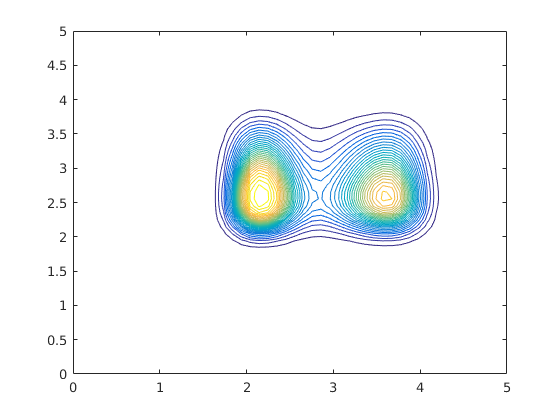}\\
surface plot  of $\nu$ for $p=1$ $$ & support of $\nu$ for $p=1$ &surface plot of $\nu$ for $p=2$ $$ & support of $\nu$ for $p=2$\\
\end{tabular}
\begin{tabular}{@{}c@{\hspace{0.5mm}}c@{\hspace{0.5mm}}@{}}
\includegraphics[ scale=0.28]{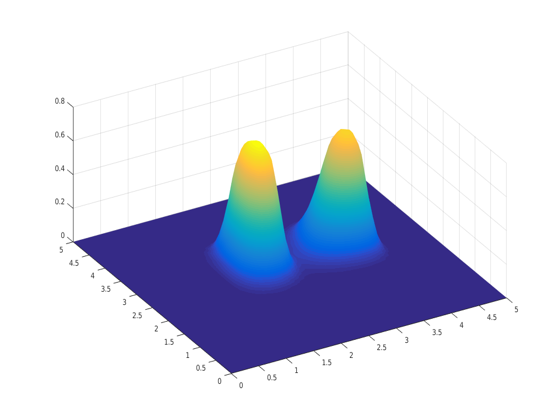}&
\includegraphics[ scale=0.32]{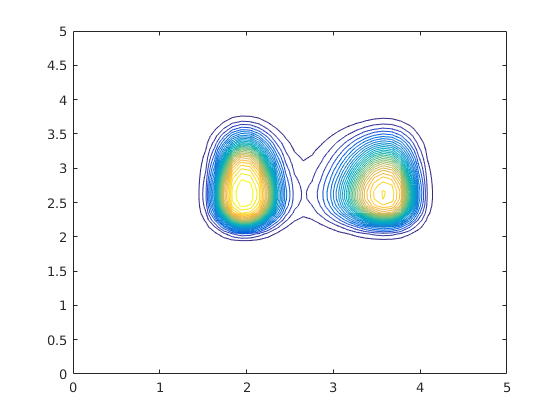}\\
surface plot  of $\nu$ for $p=4$ $$ & support of $\nu$ for $p=4$\\

\end{tabular}
\caption{\textit{The initial distribution $\mu$, a sum of two translated Gaussian, and the solution $\nu$ for different values of $p$.}}
\label{fig5}
\end{figure}

\section{Extension to several populations}\label{sec-several-pop}

\subsection{A class of two-populations models}

We end the paper by briefly explaining how our approach can easily be extended to the case of several populations of players. For the sake of simplicity, we take the two-populations case and assume that these two populations interact through a congestion term. More precisely, we are given two finite type spaces $X_1=\{x_i^1\}_{i\in I_1}$ and $X_2=\{x_i^2\}_{i\in I_2}$, a common strategy space $Y=\{y_j\}_{j\in J}$, given distributions of the players types $\mu^1\in \PP(X_1)$, $\mu^2\in \PP(X_2)$, two transport cost matrices $c^1\in \R^{I_1\times J}$, $c^2\in \R^{I_2\times J}$, and consider the minimization problem:
\begin{equation}
\inf_{(\nu^1, \nu^2)\in \PP(Y)\times \PP(Y)}\Big\{ \MK^1_{\eps_1}(\nu^1)+ \MK^2_{\eps_2}(\nu^2)+   E_1(\nu^1)+E_2(\nu^2)+ F(\nu^1+\nu^2)\Big\}
\end{equation}
where for $l=1,2$, $\eps_l>0$ is a regularization (or noise) parameter, $\MK^l_{\eps_l}(\nu^l)$ represents the regularized transport cost:
\[\MK^l_{\eps_l}(\nu^l):=\inf_{\gamma \in \Pi(\mu^l, \nu^l)}  \Big\{ c^l\cdot \gamma + \eps_l \sum_{i,j\in I_l\times J} \gamma_{ij} (\ln(\gamma_{ij})-1)\Big\},\]
$E_l(\nu^l)$ represents an individual cost for population $k$, for instance, an interaction cost:
\[E_l(\nu^l):=\sum_{j,k\in J\times J} \phi_{kj}^l \nu_j^l \nu_k^l\]
 and $F$ is a total congestion cost
 \[F(\nu^1+\nu^2):=\sum_{j\in J} F_j (\nu^1_j+\nu^2_j)\]
where $F_j$ is convex.
\begin{rem} 
The proximal step related to $F$ can be computed as in (\ref{proxconges}) by taking $\nu_j=\nu_j^1+\nu_j^2$.
\end{rem}
\subsection{Numerical Results}
For the two populations case, we take the following energies $E_l$
\[E_l(\nu^l)=\sum_{j\in J}(\nu_j^l)^8+\sum_{k,j\in J\times J}\phi_{kj}^l\nu_j^l \nu_k^l+\sum_{j\in J}|y_j-10|^4 ,\]
where $\phi_{kj}^l=10^{-4}|y_k-y_j|^2$ and the total congestion $F_j$ is given by
\[ F_j(\nu_j^1+\nu_j^2)=(\nu_j^1+\nu_j^2)^4.\]
As usual, we consider cost functions of the form $c_{ij}=|x_i-y_j|^p$ and we want to analyze the support of $\nu^l$ as $p$ varies.
For the simulations in Figure \ref{fig1_2pop} we have used  $N=500$ grid points discretization of $[0,16]$ and treated the interaction term with a semi-implicit approach.
As we can notice in Figure \ref{fig1_2pop} there is a \textit{competition} between the confinement potential and the total congestion: the two populations tend to concentrated near $y=10$ by the potential, but the effect of the congestion term makes it costly. This becomes clear if we compare (for instance, the case with $p=2$) $\nu^1$ with the optimal one in Figure \ref{fig1D_1};
even if the energies are the same, the effect of  congestion makes the support of the optimal solutions quite different.

\begin{figure}[htbp]
\label{fig1_2pop}
\begin{tabular}{@{}c@{\hspace{1.5mm}}c@{\hspace{1.5mm}}@{}}

\centering
\includegraphics[ scale=0.4]{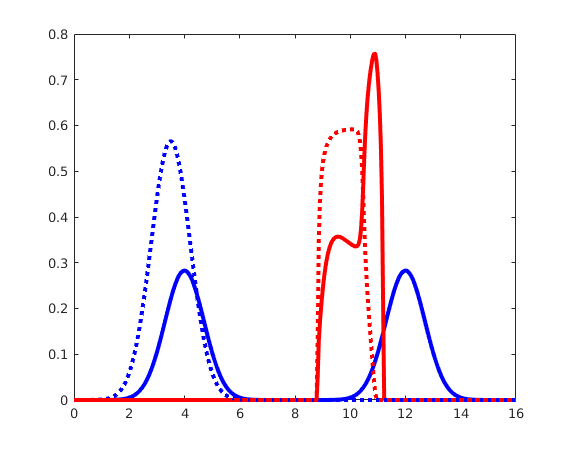}&
\includegraphics[ scale=0.4]{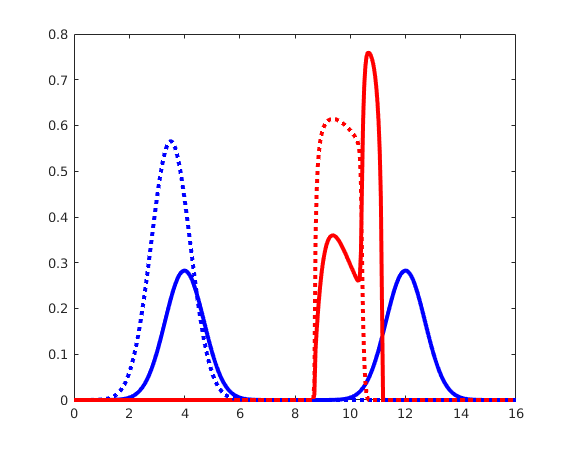}\\
$p=0.5$ & $p=1$\\
\includegraphics[ scale=0.43]{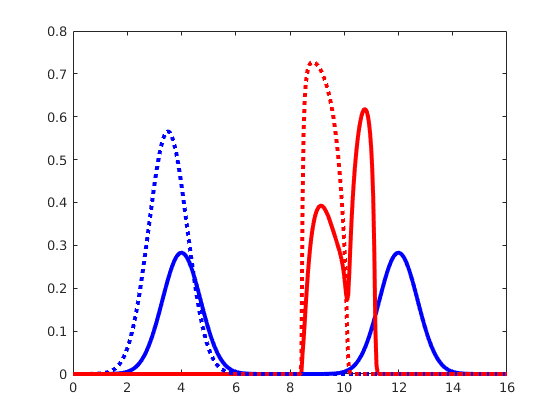}&
\includegraphics[ scale=0.4]{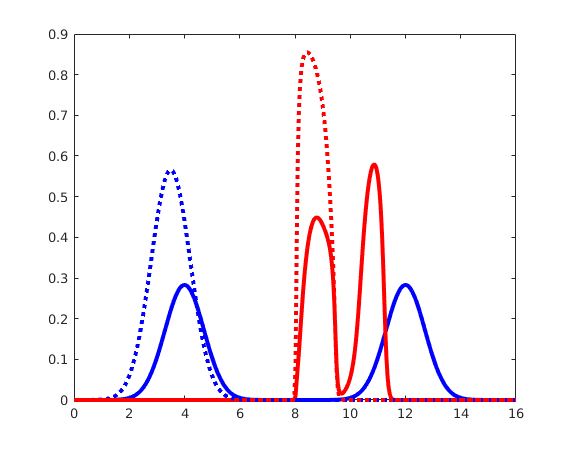}\\
$p=1.5$  & $p=2$\\

\end{tabular}
\caption{\textit{The initial distributions $\mu^1$ and $\mu^2$ (blue solid line and blue dotted line) and the solutions $\nu^1$ and $\nu^2$ (red solid line and red dotted line) for different values of $p$.}}
\label{fig1_2pop}
\end{figure}
Thus, let us now consider the following case: let $E_l$ be as above and $p=2$, then we take the total congestion given by
\[ F_j(\nu_j^1+\nu_j^2)=(\nu_j^1+\nu_j^2)^r\]
and we compute the optimal $\nu^l$ for different values of $r$.
In Figure \ref{fig2_2pop} we can see that the congestion term becomes more dominant as $r$ increases so that the two populations try to be as far as possible, despite the effect of the confinement potential which is minimal at $y=10$.
\begin{figure}[htbp]
\label{fig2_2pop}
\begin{tabular}{@{}c@{\hspace{0.5mm}}c@{\hspace{0.5mm}}c@{\hspace{0.5mm}}@{}}

\centering
\includegraphics[ scale=0.3]{1D_2bumps_1bump_pow8_quadratic_int_p2_2species.png}&
\includegraphics[ scale=0.3]{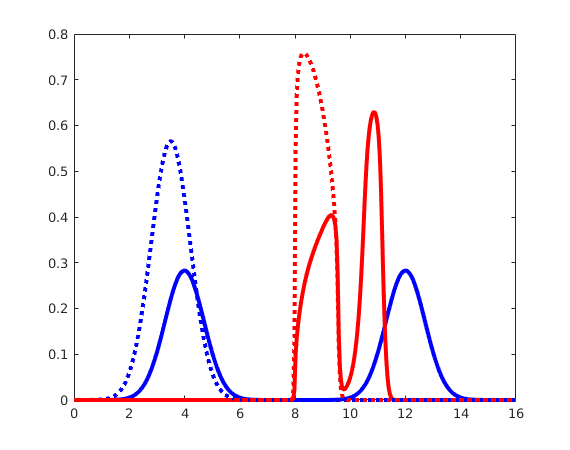}&
\includegraphics[ scale=0.3]{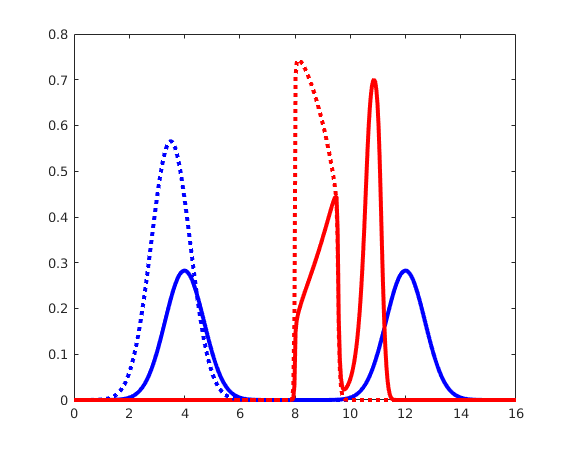}\\
$r=4$& $r=8$ & $r=32$\\

\end{tabular}
\caption{\textit{The initial distributions $\mu^1$ and $\mu^2$ (blue solid line and blue dotted line) and the solutions $\nu^1$ and $\nu^2$ (red solid line and red dotted line) for different values of $r$.}}
\label{fig2_2pop}
\end{figure}


{\bf{Acknowledgements:}}   G.C. and L.N. gratefully acknowledge the support from the ANR, through the project ISOTACE (ANR-12-
MONU-0013).

\bibliographystyle{plain}
\bibliography{bibli}

\end{document}